\documentclass[11pt]{amsart}
\usepackage{amsmath}
\usepackage{amssymb}
\usepackage{amsthm}
\usepackage{latexsym}
\usepackage{hyperref}
\usepackage{enumerate}
\usepackage[all]{xy}

\usepackage{amssymb,amsmath,amscd,url,geometry,pdflscape}

\newtheorem{thm}{Theorem}[section]
\newtheorem{cor}[thm]{Corollary}

\def\fo{\mathfrak{o}}
\def\fp{\mathfrak{p}}

\def\Hom{\mathrm{Hom}}

\def\dep{\mathrm{dep}}

\def\F{\mathbb{F}}

\def\C{\mathbb{C}}

\def\W{\mathbf{W}}

\begin{document}

\title[Depth]{Note on the non-preservation of depth}

\author[R. Plymen]{Roger Plymen}
\address{School of Mathematics, Southampton University, Southampton SO17 1BJ,  England
\emph{and} School of Mathematics, Manchester University, Manchester M13 9PL, England}
\email{r.j.plymen@soton.ac.uk \quad roger.j.plymen@manchester.ac.uk}

\date{\today}
\maketitle

\begin{abstract}  
Let $K$ be a local field of characteristic $p$.    We consider the local Langlands correspondence for tori, and construct examples for which depth is not preserved. 
\end{abstract}


\section{Introduction}  Let $K$ be a local non-archimedean field.   Let $T = R_{L/K} \mathbb{G}_m$ be an induced
torus, when $L$ is a finite separable extension of $K$.  The LLC (Local Langlands Correspondence) for tori induces an isomorphism
\[
\lambda_T : \Hom(T(K), \C^\times) \simeq \mathrm{H}^1(\W_K, T^\vee)
\]
where $\W_K$ is the Weil group of $K$ and $T^\vee$ is the complex dual torus, see \cite{MP} and \cite{Yu}.   

  For background material on depth, see \cite{ABPS}.     
 Concerning depth-preservation, we have the theorem of Yu \cite[\S 7.10]{Yu}:  In the LLC for tori, if $T$ splits over a tamely ramified extension, then we have 
 \[
 \dep(\chi) = \dep (\lambda_T(\chi)).
 \] 
 
Mishra and Patanayak \cite{MP} have recently constructed, in characteristic $0$, an explicit example of a wildly ramified torus for which depth is not preserved under LLC for all positive depth characters. 

We produce explicit examples, in characteristic $p$,  of  wildly ramified tori for which depth is not preserved under LLC for all positive depth characters.
In particular, let $K$ be a local field of characteristic $2$, and let $L/K$ be a totally ramified quadratic extension: 
there are countably many of these, with ramification breaks given by 
$m = 1, 3, 5, 7, \ldots$.  In the LLC for tori, the depth, for all positive depth characters, is not preserved.       

We wish to thank Maarten Solleveld for pointing out an error in the first version of this Note, and for several valuable comments.

\section{On depth}  
Let $K$ be a local field of characteristic $p$.  
Let $\mathfrak{o}$ be the ring of integers in $K$ and $\mathfrak{p}\subset\mathfrak{o}$ the maximal ideal. 
Let $\wp(x) = x^p - x$.   


Let $\overline{K} = K/\wp(K)$.   Let $D \neq \overline{\fo}$ be an $\F_p$-line in $\overline{K}$, $m$ the integer such that $D \subset \overline{{\fp}^{-m}}$ but 
$D\not\subset \overline{{\fp}^{-m+1}}$; we know that $m$ is $>0$ and prime to $p$.   Fix an element $a \in \fp^{-m}$ whose image generates $D$, let $\alpha$ be a root of $T^p - T - a$ (in an algebraic closure of 
$K$), and let $L = K(\alpha) = K(\wp^{-1}(D))$.   See \cite[\S 6]{Da}.   

The extension $L/K$ is totally (and wildly) ramified.   The unique ramification break of the degree $p$ cyclic extension $L/K$ occurs at $m$, see \cite[\S 6]{Da}.  Set $T = L^\times$, then $T$ is a wildly ramified
torus.   

\begin{thm}  Let $K$ be a local field of characteristic $p$, let $L = K(\wp^{-1}(D))$ as above, let $\chi$ be any character of $T$ of positive depth.      In the local Langlands correspondence for tori, the depth of the 
character $\chi$ is not preserved.     
\end{thm}

\begin{proof} We have the elegant recent formula of Mishra and Patanayak \cite{MP}:
\begin{equation}\label{mp}
\varphi_{L/K}(e \cdot \dep_T(\chi))  = \dep_{\W_K}(\lambda_T(\chi))
\end{equation}
where $\varphi_{L/K}$ is the Hasse-Herbrand function, 
 $\chi$ is a character of $T$, $\dep_T(\chi)$ is the depth of $\chi$, $\dep_{\W_K}(\lambda_T(\chi)$ is the depth of the Langlands parameter
$\lambda_T(\chi)$, and $e = e(L/K)$ is the ramification index.    If $u$ is a real number $\geq -1$, $G_u$ denotes the ramification group $G_i$, where $i$ is the smallest integer $\geq u$.   Then the Hasse-Herbrand function is
\[
\varphi_{L/K}(u) = \int_0^u\frac{1}{G_0:G_t)}dt.
\]  

\medskip

We will write $d: = \dep_T(\chi)$.

\emph{First case}.  We suppose  that  $d > m/p$.   We apply the formula (\ref{mp}):
\begin{eqnarray*}
\dep_{\W_K}(\lambda_T(\chi)) & = & \varphi_{L/K}(e \cdot \dep_T(\chi)\\
 & = & \varphi_{L/K}(pd)  \\ 
& = &\int_0^{pd} \frac {1}{(G_0 : G_t)} dt   \\ 
& = &\int_0^m 1 dt + \int_{m}^{pd} \frac{1}{p} dt    \\
& = & m +  (pd - m)/p \\
& =  & d + m(1-1/p)\\
& > & \dep_T(\chi)
\end{eqnarray*} 

\emph{Second case}. We suppose that $0 < d \leq m/p$.   We have 
\begin{eqnarray*}
\dep_{\W_K}(\lambda_T(\chi)) & = & \varphi_{L/K}(e \cdot \dep_T(\chi)\\
 & = & \varphi_{L/K}(pd)  \\ 
& = &\int_0^{pd} \frac {1}{(G_0 : G_t)} dt   \\ 
& =  & pd\\
& >  & \dep_T(\chi)
\end{eqnarray*}    
\end{proof}

\begin{cor}  Let $K$ be a local field of characteristic $2$, and let $L/K$ be a totally ramified quadratic extension: there are countably many of these, with ramification breaks given by 
$m = 1, 3, 5, 7, \ldots$.  In the LLC for tori, the depth, for all positive depth characters, is not preserved.     
\end{cor}

\end{document}